\PassOptionsToPackage{backref=page}{hyperref}
\documentclass[11pt,letterpaper,reqno]{amsart}

\usepackage[T1]{fontenc}
\usepackage{amsmath,amssymb,amsfonts,amsthm}
\usepackage{mathtools}
\usepackage{aliascnt}
\usepackage{microtype}
\usepackage{enumitem}
\usepackage{booktabs}
\usepackage{xcolor}
\usepackage{doi}
\usepackage{hyperref}
\usepackage{tikz}
\usetikzlibrary{positioning,arrows.meta}
\usepackage{bookmark}
\usepackage[capitalize,noabbrev]{cleveref}

\renewcommand*{\backref}[1]{}
\renewcommand*{\backrefalt}[4]{%
  \ifcase #1
  \or
    \quad$\hookleftarrow$, cited on page~#2%
  \else
    \quad$\hookleftarrow$, cited on pages~#2%
  \fi
}

\newtheorem{theorem}{Theorem}[section]

\newaliascnt{lemma}{theorem}
\newtheorem{lemma}[lemma]{Lemma}
\aliascntresetthe{lemma}

\newaliascnt{proposition}{theorem}
\newtheorem{proposition}[proposition]{Proposition}
\aliascntresetthe{proposition}

\newaliascnt{corollary}{theorem}
\newtheorem{corollary}[corollary]{Corollary}
\aliascntresetthe{corollary}

\newaliascnt{conjecture}{theorem}
\newtheorem{conjecture}[conjecture]{Conjecture}
\aliascntresetthe{conjecture}

\newaliascnt{problem}{theorem}

\aliascntresetthe{problem}

\newaliascnt{question}{theorem}

\aliascntresetthe{question}

\theoremstyle{definition}
\newaliascnt{definition}{theorem}

\aliascntresetthe{definition}

\newaliascnt{example}{theorem}

\aliascntresetthe{example}

\newaliascnt{remark}{theorem}

\aliascntresetthe{remark}

\crefname{theorem}{Theorem}{Theorems}
\Crefname{theorem}{Theorem}{Theorems}
\crefname{lemma}{Lemma}{Lemmas}
\Crefname{lemma}{Lemma}{Lemmas}
\crefname{proposition}{Proposition}{Propositions}
\Crefname{proposition}{Proposition}{Propositions}
\crefname{corollary}{Corollary}{Corollaries}
\Crefname{corollary}{Corollary}{Corollaries}
\crefname{conjecture}{Conjecture}{Conjectures}
\Crefname{conjecture}{Conjecture}{Conjectures}

\newcommand{\R}{\mathbb{R}}
\newcommand{\dd}{\,\mathrm{d}}

\begin{document}

\title[Higher-index Dirichlet eigenvalue ratios]{Counterexamples to a higher-index Dirichlet eigenvalue-ratio conjecture}

\author[Y.~He]{Yixin He}
\address{School of Mathematical Sciences, Fudan University,
Shanghai 200433, P.~R.~China}
\email{yixin.he717@gmail.com}

\author[Q.~Tang]{Quanyu Tang}
\address{School of Mathematical Sciences, University of Science and Technology
of China, Hefei 230026, P.~R.~China}
\email{tangquanyu827@gmail.com}

\author[H.~Zhang]{Haiqi Zhang}
\address{School of Mathematics, Shandong University, Jinan 250100, P.~R.~China}
\email{ZHQAQ2024@outlook.com}

\begin{abstract}
Let $\lambda_k(\Omega)$ denote the $k$th Dirichlet Laplacian eigenvalue of a bounded planar domain $\Omega$, with eigenvalues counted with multiplicity.  We disprove a conjectured higher-index extension of the Payne--P\'olya--Weinberger inequality that was recorded as an open problem by Ashbaugh. For every integer $m\geq3$, we construct a bounded planar domain $\Omega_m$ with $C^\infty$ boundary such that
$$
\frac{\lambda_{2m}(\Omega_m)}{\lambda_m(\Omega_m)}
>\frac{13}{5}
>\frac{j_{1,1}^2}{j_{0,1}^2},
$$
where $j_{\nu,1}$ is the first positive zero of $J_\nu$, the Bessel function of the first kind of order $\nu$. For $m=3$, the annulus $\{x\in\R^2:1/10<|x|<1\}$ already provides such a counterexample.
\end{abstract}

% 35P15  	Estimates of eigenvalues in context of PDEs
% 58J50  	Spectral problems; spectral geometry; scattering theory on manifolds


\maketitle

\section{Introduction}

Throughout the paper, a \emph{domain} is a bounded connected open subset of
Euclidean space. We write $\mathbb N=\{1,2,\ldots\}$ and
$\mathbb N_0=\{0,1,2,\ldots\}$. For any bounded open set
$G\subset\R^N$, not necessarily connected, let $\lambda_j(G)$ denote the
$j$th eigenvalue of the Dirichlet Laplacian, with eigenvalues counted with
multiplicity. In particular, if $G$ is a domain, then
$0<\lambda_1(G)<\lambda_2(G)\leq\lambda_3(G)\leq\cdots$.
If $J_\nu$ denotes the Bessel function of the first kind of order $\nu$
and $j_{\nu,q}$ its $q$th positive zero, define
\[
 \gamma_N:=\frac{j_{N/2,1}^2}{j_{N/2-1,1}^2}.
\]
Payne, P\'olya, and Weinberger initiated the systematic study of universal
ratios of Dirichlet eigenvalues.  In the planar case they conjectured that
the ratio $\lambda_2/\lambda_1$ is maximized by the disk, and their 1956
paper established classical non-sharp estimates that motivated the problem
\cite{PPW1956}.  The corresponding sharp inequality in dimension $N$ is
\begin{equation}\label{eq:ppw}
 \frac{\lambda_2(\Omega)}{\lambda_1(\Omega)}\leq\gamma_N.
\end{equation}
Ashbaugh and Benguria proved \eqref{eq:ppw} in all dimensions
\cite{AshbaughBenguria1992}; the constant is sharp and is attained by balls.

The higher-index question considered here grew out of the same circle of
ideas.  Ashbaugh and Benguria proved, in particular, that
$\lambda_4(\Omega)/\lambda_2(\Omega)<\gamma_N$ for connected domains, as a
special case of a more general estimate involving the number of nodal domains
of an eigenfunction \cite{AshbaughBenguria1993}.  They also obtained the
iterated bound
$\lambda_{2^r}(\Omega)/\lambda_1(\Omega)\leq\gamma_N^r$
\cite{AshbaughBenguria1994}.  Ashbaugh subsequently included the corresponding
one-step inequality in his list of open problems
\cite[Problem~3, p.~15]{Ashbaugh1999}.  The same problem was recorded in the
2007 and 2009 Oberwolfach reports.  Those reports noted that the cases
$m=1,2$ were known and that the cases $m\geq3$ remained open
\cite[p.~1021]{BucurButtazzoHenrot2007}; see also
\cite[pp.~400--401]{AshbaughBenguriaLaugesenWeidl2009}.
The conjecture can be stated as follows.

\begin{conjecture}\label{conj:ashbaugh}
For every integer $m\geq1$, every $N\geq2$, and every domain
$\Omega\subset\R^N$,
\begin{equation}\label{eq:ashbaugh}
 \frac{\lambda_{2m}(\Omega)}{\lambda_m(\Omega)}\leq\gamma_N.
\end{equation}
\end{conjecture}

Our main result disproves the conjecture for every $m\geq3$, already in the
plane.

\begin{theorem}\label{thm:main}
For every integer $m\geq3$, there exists a planar domain $\Omega_m$ with
$C^\infty$ boundary such that
\[
 \frac{\lambda_{2m}(\Omega_m)}{\lambda_m(\Omega_m)}>\frac{13}{5}.
\]
For $m=3$, one may take the circular annulus $A:=\bigl\{x\in\R^2:1/10<|x|<1\bigr\}$.
\end{theorem}

Lemma~\ref{lem:bessel-certificates} gives $\gamma_2<13/5$.
Consequently, Theorem~\ref{thm:main} gives strict counterexamples to
\eqref{eq:ashbaugh}.  It also disproves, for $N=2$ and $k=2$, the natural
more general estimate
\[
 \frac{\lambda_{km}(\Omega)}{\lambda_m(\Omega)}\leq C_{N,k},
 \qquad
 C_{N,k}:=\sup_D\frac{\lambda_k(D)}{\lambda_1(D)},
\]
where the supremum is taken over domains $D\subset\R^N$.  This broader
formulation was also discussed in Ashbaugh's open-problem survey
\cite[p.~15]{Ashbaugh1999}.  Recent non-sharp estimates for ratios of arbitrary
Dirichlet eigenvalues on convex domains may be found in \cite{Funano2026};
our examples are nonconvex and do not address possible restrictions of
Conjecture~\ref{conj:ashbaugh} to narrower geometric classes.

The proof is explicit. Separation of variables reduces the annular spectrum
to one-dimensional weighted problems. Polynomial trial functions and a
localized Hardy estimate identify the first seven annular eigenvalues and
show that $\lambda_6(A)/\lambda_3(A)>13/5$. All numerical comparisons are
certified by exact arithmetic and rational bounds, including Euler--Rayleigh
bounds for Bessel zeros, so no floating-point computation enters the proof.
To treat every $m\geq4$, we adjoin $m-3$ rectangles whose first three
eigenvalues interlace with the relevant annular eigenvalues in the required
way. Exact spectral counting on the disconnected union gives the same ratio,
and standard spectral convergence under shrinking connections then produces
smooth connected examples.

Section~\ref{sec:bessel} develops the Bessel-zero certificates.
Section~\ref{sec:annulus} analyzes the annulus and proves the case $m=3$.
Section~\ref{sec:all-m} amplifies the construction to every $m\geq4$.
Section~\ref{sec:nodal} explains why the usual nodal-domain argument stops
short of the conjectured index.  Exact trial-function integrals and the
remaining elementary rational estimates are collected in the appendix.

\section{Certified bounds for Bessel zeros}\label{sec:bessel}

We use the following elementary form of the Euler--Rayleigh inequalities;
see, for example, \cite[Lemma~3.2]{IsmailMuldoon1995}. The product formula for \(J_\nu\) and the Rayleigh sums of its zeros are standard; see \cite[\S10.21(iii),(xiii)]{DLMF}.

\begin{lemma}\label{lem:euler-rayleigh}
Let $x_1>x_2>\cdots>0$ and suppose that
$S_p:=\sum_{q\geq1}x_q^p<\infty$.  Then, for every integer $p\geq1$,
\[
 S_p^{-1/p}<\frac1{x_1}<\frac{S_p}{S_{p+1}}.
\]
\end{lemma}

\begin{proof}
The first inequality follows from $S_p>x_1^p$.  Since $x_q<x_1$ for
$q\geq2$, one also has $S_{p+1}<x_1S_p$, which gives the second inequality.
\end{proof}

For $\nu\geq0$ and $p\geq1$, define the \emph{Rayleigh sum}
$\sigma_p(\nu):=\sum_{q=1}^\infty j_{\nu,q}^{-2p}$.  Applying
Lemma~\ref{lem:euler-rayleigh} to $x_q=j_{\nu,q}^{-2}$ gives
\begin{equation}\label{eq:euler-rayleigh-bessel}
 \sigma_p(\nu)^{-1/p}<j_{\nu,1}^2
 <\frac{\sigma_p(\nu)}{\sigma_{p+1}(\nu)}.
\end{equation}
The Rayleigh sums satisfy the classical convolution recurrence of
Kishore~\cite{Kishore1963}; see also the Riccati-equation derivation and
generalizations of Gupta and Muldoon~\cite{GuptaMuldoon2000}:
\begin{equation}\label{eq:rayleigh-recursion}
 \sigma_1(\nu)=\frac{1}{4(\nu+1)},
 \qquad
 \sigma_p(\nu)=\frac{1}{\nu+p}
 \sum_{r=1}^{p-1}\sigma_r(\nu)\sigma_{p-r}(\nu)
 \quad(p\geq2).
\end{equation}
The values needed below are listed in Table~\ref{tab:rayleigh-sums}; each
follows by repeated substitution in \eqref{eq:rayleigh-recursion}.
These computations may also be verified by the accompanying exact-arithmetic
script~\cite{ExactArithmeticCode}.

\begin{table}[ht]
\centering
\caption{Rayleigh sums used in the proof.}\label{tab:rayleigh-sums}
\begin{tabular}{ccl}
\toprule
$\nu$ & $p$ & $\sigma_p(\nu)$ \\
\midrule
$0$ & $2$ & $1/32$ \\
$1$ & $3$ & $1/3072$ \\
$1$ & $5$ & $13/8847360$ \\
$1$ & $6$ & $11/110100480$ \\
$3$ & $5$ & $1/110100480$ \\
\bottomrule
\end{tabular}
\end{table}

\begin{lemma}\label{lem:bessel-certificates}
The first Bessel zeros satisfy
\[
 j_{1,1}^2>14,
 \qquad
 j_{3,1}^2>\frac{1014}{25},
 \qquad
 \gamma_2=\frac{j_{1,1}^2}{j_{0,1}^2}<\frac{13}{5}.
\]
\end{lemma}

\begin{proof}
The lower bound in \eqref{eq:euler-rayleigh-bessel} gives
$j_{1,1}^2>3072^{1/3}>14$, since $3072>14^3$.  It also gives
\[
 j_{3,1}^2>110100480^{1/5}>\frac{1014}{25},
\]
because
$110100480\cdot25^5-1014^5=3212367382176>0$.

For the ratio, the lower and upper bounds in
\eqref{eq:euler-rayleigh-bessel} give
\[
 j_{0,1}^2>\sqrt{32},
 \qquad
 j_{1,1}^2<\frac{\sigma_5(1)}{\sigma_6(1)}=\frac{1456}{99}.
\]
Hence
\[
 \gamma_2=\frac{j_{1,1}^2}{j_{0,1}^2}
 <\frac{1456}{99\sqrt{32}}<\frac{13}{5},
\]
where the last inequality follows after squaring from
$13^2\cdot99^2\cdot32-5^2\cdot1456^2=5408>0$.
\end{proof}

\section{An annular counterexample}\label{sec:annulus}

Set $a:=1/10$ and $A:=\{x\in\R^2:a<|x|<1\}$.  For
$n\in\mathbb{N}_0$ and $f\in H_0^1(a,1)\setminus\{0\}$, define
\begin{equation}\label{eq:radial-quotient}
 \mathcal{Q}_n[f]:=
 \frac{\displaystyle\int_a^1
 \left(r|f'(r)|^2+\frac{n^2}{r}|f(r)|^2\right)\dd r}
 {\displaystyle\int_a^1 r|f(r)|^2\dd r}.
\end{equation}
This is precisely the Rayleigh quotient of the Dirichlet Laplacian on
functions of the form $u(r,\theta)=f(r)e^{in\theta}$; see, for example,
\cite[(2.6) and Appendix~A]{AnoopBobkovDrabek2022}.
Let $\mu_{n,q}$ be the $q$th variational eigenvalue associated with
\eqref{eq:radial-quotient}.  Separation of variables then shows that
the Dirichlet spectrum of $A$ is the multiset consisting of $\mu_{0,q}$ with
multiplicity one and $\mu_{n,q}$ with multiplicity two for every $n\geq1$.

\subsection{Angular monotonicity and comparison with the disk}

\begin{lemma}\label{lem:angular-monotonicity}
For every $n\geq0$ and $q\geq1$,
\begin{equation}\label{eq:angular-gap}
 \mu_{n+1,q}\geq\mu_{n,q}+2n+1.
\end{equation}
\end{lemma}

\begin{proof}
Since $a<r<1$, one has $r^{-1}\geq r$. Hence, for every
$f\in H_0^1(a,1)\setminus\{0\}$,
\[
 \mathcal{Q}_{n+1}[f]-\mathcal{Q}_n[f]
 =(2n+1)\frac{\int_a^1|f|^2r^{-1}\dd r}
 {\int_a^1|f|^2r\dd r}
 \geq2n+1.
\]
Taking the min--max over \(q\)-dimensional subspaces of \(H_0^1(a,1)\) gives \eqref{eq:angular-gap}.
\end{proof}

\begin{lemma}\label{lem:disk-sector}
For every $n\geq0$, one has $\mu_{n,1}\geq j_{n,1}^2$.
\end{lemma}

\begin{proof}
Extend \(f\in H_0^1(a,1)\) by zero from \((a,1)\) to \((0,1)\).
This embeds the form domain of the annular $n$th angular sector into the
corresponding form domain for the unit disk.  The min--max principle gives
$\mu_{n,1}\geq j_{n,1}^2$.
\end{proof}

\subsection{Separation from the second radial mode}

The only step in identifying the low-lying spectrum that does not follow
immediately from monotonicity is to show that the first $n=3$ mode lies
below the second radial mode.

\begin{proposition}\label{prop:sector-separation}
One has $\mu_{3,1}<\mu_{0,2}$.
\end{proposition}

\begin{proof}
Consider the trial function
$f_3(r):=(r-a)(1-r)(1-7r+4r^2)$. The exact integral evaluations recorded in Appendix~\ref{app:trial-integrals} give
\[
 \mathcal{Q}_3[f_3]
 =\frac{7000\bigl(10^6\log 10+12031677\bigr)}{2451519837}.
\]
Using $\log 10<2303/1000$ from
Lemma~\ref{lem:elementary-constants}, we obtain
\begin{equation}\label{eq:mu31-upper}
 \mu_{3,1}\leq\mathcal{Q}_3[f_3]
 <\frac{100342739000}{2451519837}<41.
\end{equation}

We next estimate the second radial eigenvalue from below.  Put
$v(r):=r^{1/2}f(r)$.  Integration by parts, using $v(a)=v(1)=0$, gives
\[
 \int_a^1r|f'|^2\dd r
 =\int_a^1\left(|v'|^2-\frac{|v|^2}{4r^2}\right)\dd r,
 \qquad
 \int_a^1r|f|^2\dd r=\int_a^1|v|^2\dd r.
\]
Choose $c:=391/1000$.  From
$v(r)=\int_a^rv'(s)\dd s$ on $(a,c)$ and
$v(r)=-\int_r^1v'(s)\dd s$ on $(c,1)$, Cauchy--Schwarz and Fubini's theorem
give
\begin{align*}
 \int_a^c\frac{|v|^2}{4r^2}\dd r
 &\leq C_L\int_a^c|v'|^2\dd r,
 &C_L&:=\frac14\left(\log\frac ca+\frac ac-1\right),\\
 \int_c^1\frac{|v|^2}{4r^2}\dd r
 &\leq C_R\int_c^1|v'|^2\dd r,
 &C_R&:=\frac14\left(\frac1c+\log c-1\right).
\end{align*}
The logarithmic estimates in Lemma~\ref{lem:elementary-constants} imply
$C_L<39/250$ and $C_R<39/250$.  Therefore
\begin{equation}\label{eq:hardy-lower}
 \int_a^1\left(|v'|^2-\frac{|v|^2}{4r^2}\right)\dd r
 >\frac{211}{250}\int_a^1|v'|^2\dd r.
\end{equation}
The min--max principle and the Dirichlet spectrum of the interval $(a,1)$
now give
\[
 \mu_{0,2}\geq\frac{211}{250}
 \left(\frac{2\pi}{1-a}\right)^2
 =\frac{211}{250}\left(\frac{20\pi}{9}\right)^2.
\]
Using $\pi>157/50$ from Lemma~\ref{lem:elementary-constants}, we conclude
that
\begin{equation}\label{eq:mu02-lower}
 \mu_{0,2}>\frac{10401878}{253125}>41.
\end{equation}
The bounds \eqref{eq:mu31-upper} and \eqref{eq:mu02-lower} prove the claim.
\end{proof}

\begin{corollary}
The first seven eigenvalues of $A$ are
\begin{equation}\label{eq:low-annular-spectrum}
\begin{aligned}
 \lambda_1(A)&=\mu_{0,1},
 &\lambda_2(A)=\lambda_3(A)&=\mu_{1,1},\\
 \lambda_4(A)=\lambda_5(A)&=\mu_{2,1},
 &\lambda_6(A)=\lambda_7(A)&=\mu_{3,1}.
\end{aligned}
\end{equation}
\end{corollary}

\begin{proof}
Lemma~\ref{lem:angular-monotonicity} gives
$\mu_{0,1}<\mu_{1,1}<\mu_{2,1}<\mu_{3,1}<\mu_{4,1}<\cdots$.  For every
$n\geq0$ and $q\geq2$, angular monotonicity gives
$\mu_{n,q}\geq\mu_{0,q}\geq\mu_{0,2}>\mu_{3,1}$, where the last inequality
is Proposition~\ref{prop:sector-separation}.  Taking angular multiplicities
into account yields \eqref{eq:low-annular-spectrum}.
\end{proof}

\subsection{Eigenvalue-ratio estimate for the annulus}

Define
$\alpha:=\mu_{1,1}=\lambda_3(A)$ and
$\beta:=\mu_{3,1}=\lambda_6(A)$.

\begin{proposition}\label{prop:annular-ratio}
The annulus $A$ satisfies
\[
 \frac{\lambda_6(A)}{\lambda_3(A)}
 =\frac{\beta}{\alpha}>\frac{13}{5}.
\]
\end{proposition}

\begin{proof}
Lemmas~\ref{lem:disk-sector} and~\ref{lem:bessel-certificates} give
$\beta=\mu_{3,1}\geq j_{3,1}^2>1014/25$.  For the denominator, take
$f_1(r):=(r-a)(1-r)(1-r/2)$.  Appendix~\ref{app:trial-integrals} gives
\[
 \mathcal{Q}_1[f_1]
 =\frac{56\bigl(800000\log 10+4831011\bigr)}{23993577}.
\]
The estimate $\log 10<2303/1000$ yields
\[
 \alpha\leq\mathcal{Q}_1[f_1]
 <\frac{373711016}{23993577}<\frac{78}{5}.
\]
It follows that
\[
 \frac{\beta}{\alpha}
 >\frac{1014/25}{78/5}=\frac{13}{5}.
\qedhere
\]
\end{proof}

Lemma~\ref{lem:bessel-certificates} gives $\gamma_2<13/5$, so
Proposition~\ref{prop:annular-ratio} proves Theorem~\ref{thm:main} for
$m=3$.

\section{Counterexamples for all higher indices}\label{sec:all-m}

We first record the standard approximation that turns a finite disjoint union
into a smooth connected domain without changing any fixed finite part of its
Dirichlet spectrum in the limit.

\begin{lemma}\label{lem:connected-approximation}
Let $G^{(0)}\subset\R^2$ be a finite disjoint union of bounded Lipschitz
domains.  There exists a sequence of bounded connected domains $G_\ell$ with
$C^\infty$ boundary such that, for every fixed $j\geq1$,
\[
 \lambda_j(G_\ell)\longrightarrow\lambda_j(G^{(0)})
 \qquad(\ell\to\infty).
\]
\end{lemma}

\begin{proof}
Write \(G^{(0)}=D_1\mathbin{\dot\cup}\cdots\mathbin{\dot\cup}D_s\). Choose connected smooth domains \(D_i^{(r)}\subset \subset D_i\) increasing to \(D_i\), and set
\[
G^{(r)}:=\mathbin{\dot\bigcup}_{i=1}^s D_i^{(r)}.
\]
Standard domain-approximation results give
\[
\lambda_j(G^{(r)})\longrightarrow\lambda_j(G^{(0)})
\qquad (r\to\infty)
\]
for every fixed \(j\); see
\cite[Proposition~7.1, Corollary~4.7, and Remark~4.3]{Daners2003}.

Fix \(r\). After rigid motions of the components, which do not change
their spectra, arrange them in a chain with pairwise disjoint interiors,
so that consecutive closures meet at finitely many points and all other
closures are disjoint. Let \(E\) be the resulting finite set of
contact points.  By smoothing the contacts inside small neighborhoods of
\(E\), we obtain bounded connected \(C^\infty\) domains
\(G^{(r)}_\varepsilon\), all contained in a fixed bounded set, such that
\begin{equation}\label{eq:handle-inclusions}
 G^{(r)}\subset G^{(r)}_\varepsilon
 \subset G^{(r)}\cup N_{C\varepsilon}(E),
 \qquad
 N_\rho(E):=\{x:\operatorname{dist}(x,E)<\rho\},
\end{equation}
where \(C\) is independent of \(\varepsilon\).

We verify the three hypotheses of \cite[Theorem~7.5]{Daners2003} along
an arbitrary sequence \(\varepsilon_n\downarrow0\).  First, if \(K\subset\subset G^{(r)}\), then
\(K\setminus G^{(r)}_{\varepsilon_n}=\varnothing\) by the first
inclusion in \eqref{eq:handle-inclusions}, and therefore
\[
 \operatorname{cap}
 \bigl(K\setminus G^{(r)}_{\varepsilon_n}\bigr)=0.
\]
Here \(\operatorname{cap}\) denotes the standard \((1,2)\)-capacity in
\(\mathbb R^2\),
\[
 \operatorname{cap}(F)
 :=
 \inf\left\{
   \|v\|_{H^1(\mathbb R^2)}^2:
   v\in H^1(\mathbb R^2),\
   v\geq1 \text{ a.e.\ in a neighborhood of }F
 \right\},
\]
as in \cite[p.~605]{Daners2003}. For the second
hypothesis, the single open set $U:=\mathbb R^2\setminus\overline{G^{(r)}}$ covers \(\mathbb R^2\setminus\overline{G^{(r)}}\).
Since \(E\) is finite, the second inclusion in
\eqref{eq:handle-inclusions} gives
\[
 \bigl|U\cap G^{(r)}_{\varepsilon_n}\bigr|
 \leq \bigl|N_{C\varepsilon_n}(E)\bigr|
 =O(\varepsilon_n^2)\longrightarrow0.
\]
Therefore, by \cite[Proposition~7.6(1)]{Daners2003},
\[
 \lambda_1\bigl(U\cap G^{(r)}_{\varepsilon_n}\bigr)
 \longrightarrow\infty.
\]
Finally, the boundary accumulation set appearing in
\cite[(7.2)]{Daners2003}, namely
\[
 \Gamma:=\bigcap_{n=1}^{\infty}
 \overline{\bigcup_{k\geq n}
 \bigl(G^{(r)}_{\varepsilon_k}\cap\partial G^{(r)}\bigr)},
\]
is contained in \(E\) by \eqref{eq:handle-inclusions}. Since \(\Gamma\subset E\) and \(E\) is finite, we have
\(\operatorname{cap}(\Gamma)=0\).  Consequently,
\[
 H_0^1\bigl(G^{(r)}\cup\Gamma\bigr)
 =
 H_0^1\bigl(G^{(r)}\bigr),
\]
which is the third hypothesis of
\cite[Theorem~7.5]{Daners2003}. Therefore, the convergence conclusion of \cite[Theorem~7.5]{Daners2003} applies to
\(G^{(r)}_{\varepsilon}\) as \(\varepsilon\downarrow 0\). Since all these domains lie in a fixed bounded
set, \cite[Corollary~4.7 and Remark~4.3]{Daners2003} yields
\[
 \lambda_j(G^{(r)}_\varepsilon)
 \longrightarrow\lambda_j(G^{(r)})
 \qquad (\varepsilon\downarrow0)
\]
for every fixed \(j\), counting multiplicities.
A diagonal choice of \(r\to\infty\) and \(\varepsilon\downarrow0\)
completes the proof.
\end{proof}

We now choose an auxiliary rectangle with exactly one eigenvalue below
$\alpha$ and exactly two below $\beta$.  Let
\[
 P:=\left(0,\frac{\pi}{2}\right)
 \times\left(0,\frac{\pi}{\sqrt{10}}\right).
\]
Its separated Dirichlet eigenvalues are $4p^2+10q^2$, where $p,q\in\mathbb{N}$, and hence
\[
 \lambda_1(P)=14,
 \qquad
 \lambda_2(P)=26,
 \qquad
 \lambda_3(P)=44.
\]
The estimates above give the transparent chain
\begin{equation}\label{eq:interlacing}
 14<\alpha<\frac{78}{5}<26
 <\frac{1014}{25}<\beta<41<44.
\end{equation}
In particular,
$\lambda_1(P)<\alpha<\lambda_2(P)<\beta<\lambda_3(P)$.

\begin{proposition}\label{prop:all-indices}
For every integer $m\geq4$, there exists a planar domain $\Omega_m$ with
$C^\infty$ boundary such that
\[
 \frac{\lambda_{2m}(\Omega_m)}{\lambda_m(\Omega_m)}>\frac{13}{5}.
\]
\end{proposition}

\begin{proof}
Let $P_1,\ldots,P_{m-3}$ be disjoint copies of $P$ and set
\[
 G_m^{(0)}:=A\mathbin{\dot\cup}P_1\mathbin{\dot\cup}\cdots
 \mathbin{\dot\cup}P_{m-3}.
\]
The spectrum of a finite disjoint union is the multiset union of the spectra
of its components.  By \eqref{eq:low-annular-spectrum} and
\eqref{eq:interlacing}, there are exactly
$1+(m-3)=m-2$ eigenvalues of $G_m^{(0)}$ strictly below $\alpha$, while
$\alpha$ has multiplicity two.  Hence $\lambda_m(G_m^{(0)})=\alpha$.
Likewise, there are exactly $5+2(m-3)=2m-1$ eigenvalues strictly below
$\beta$, while $\beta$ has multiplicity two.  Thus
\begin{equation}\label{eq:disconnected-ratio}
 \lambda_m(G_m^{(0)})=\alpha,
 \qquad
 \lambda_{2m}(G_m^{(0)})=\beta,
 \qquad
 \frac{\lambda_{2m}(G_m^{(0)})}{\lambda_m(G_m^{(0)})}>\frac{13}{5}.
\end{equation}
Apply Lemma~\ref{lem:connected-approximation} to $G_m^{(0)}$.  The $m$th and
$2m$th eigenvalues of the resulting smooth connected domains converge to
$\alpha$ and $\beta$, respectively. Since the last inequality in \eqref{eq:disconnected-ratio} is strict,
the same strict inequality holds for all sufficiently large members of the
approximating sequence. Choose one such domain and denote it by $\Omega_m$.
\end{proof}

Propositions~\ref{prop:annular-ratio} and~\ref{prop:all-indices} complete the
proof of Theorem~\ref{thm:main}.

\section{Relation to nodal-domain estimates}\label{sec:nodal}

The mechanism behind the known $m=2$ case also explains why the conjecture
was plausible.  Suppose that an eigenfunction $u$ for $\lambda_m(\Omega)$
has $\ell(u)$ nodal domains.  On every nodal domain $D_i$, the restriction of
$u$ is a first Dirichlet eigenfunction, so
$\lambda_1(D_i)=\lambda_m(\Omega)$.  Applying the definition of $C_{N,k}$ to
each $D_i$, taking the disjoint union, and using Dirichlet domain monotonicity
gives
\begin{equation}\label{eq:nodal-bound}
 \lambda_{k\ell(u)}(\Omega)
 \leq C_{N,k}\lambda_m(\Omega).
\end{equation}
This argument, in a sharper Bessel-zero form, is already present in
Ashbaugh and Benguria's work on higher eigenvalue ratios
\cite{AshbaughBenguria1993}.

For $m=2$, every eigenfunction associated with $\lambda_2(\Omega)$ on a
connected domain has exactly two nodal domains, and
\eqref{eq:nodal-bound} with $k=2$ gives the required index $4$. The annulus makes the obstruction at $m=3$ transparent.  Its
$\lambda_3$-eigenspace is
$f(r)\operatorname{span}\{\cos\theta,\sin\theta\}$, where $f$ is the
positive first radial eigenfunction in the $n=1$ sector.  Every nonzero
eigenfunction in this space has exactly two nodal domains, not three.  Thus
the nodal argument controls $\lambda_4(A)$ at the scale
$\gamma_2\lambda_3(A)$, whereas Conjecture~\ref{conj:ashbaugh} asks for
control of $\lambda_6(A)$.  Proposition~\ref{prop:annular-ratio} shows that
this missing index shift is not merely a limitation of the method: the
claimed inequality itself fails.

\appendix
\section{Exact-arithmetic verification}

For reproducibility, the exact computations in this appendix, as well as the
Rayleigh sums in Table~\ref{tab:rayleigh-sums}, can be independently verified using the accompanying Python/SymPy script~\cite{ExactArithmeticCode}.
The repository also contains the output of a successful run.
The script verifies the trial-function integrals and the rational certificates
used below using symbolic and exact rational arithmetic only; no floating-point
computation is involved.

\subsection{Exact trial-function integrals}\label{app:trial-integrals}

For \(f\in H_0^1(a,1)\), set
\begin{align*}
 I_0(f)&:=\int_a^1r f(r)^2\dd r,
 &I_1(f)&:=\int_a^1r f'(r)^2\dd r,\\
 I_{-1}(f)&:=\int_a^1\frac{f(r)^2}{r}\dd r.
\end{align*}
For the two trial functions used above, direct expansion and integration give
\begin{align*}
 I_0(f_1)&=\frac{23993577}{4480000000},
 &I_1(f_1)&=\frac{2401569}{40000000},\\
 I_{-1}(f_1)&=\frac{\log 10}{100}+\frac{27873}{80000000},\\[2mm]
 I_0(f_3)&=\frac{22063678533}{700000000000},
 &I_1(f_3)&=\frac{187046577}{350000000},\\
 I_{-1}(f_3)&=\frac{\log 10}{100}+\frac{42655833}{700000000}.
\end{align*}
Substitution into
$\mathcal{Q}_n[f]=(I_1(f)+n^2I_{-1}(f))/I_0(f)$ gives the quotient formulas
used in Propositions~\ref{prop:sector-separation} and~\ref{prop:annular-ratio}.

\subsection{Elementary bounds for constants}

\begin{lemma}\label{lem:elementary-constants}
The following inequalities hold:
\[
 \log 10<\frac{2303}{1000},
 \qquad
 \pi>\frac{157}{50},
 \qquad
 \log\frac{391}{100}<\frac{273}{200},
 \qquad
 \log\frac{391}{1000}< -\frac{187}{200}.
\]
\end{lemma}

\begin{proof}
The first inequality follows from the Taylor series of the exponential:
\[
 \exp\left(\frac{2303}{1000}\right)
 >\sum_{q=0}^{9}\frac{(2303/1000)^q}{q!}>10.
\]
For the second, Machin's formula and the alternating series for the arctangent
give
\[
 \frac\pi4
 =4\arctan\frac15-\arctan\frac1{239}
 >4\left(\frac15-\frac{1}{3\cdot5^3}\right)-\frac1{239}
 =\frac{70369}{89625}>\frac{157}{200}.
\]

For $x>0$, put $t:=(x-1)/(x+1)$.  Then
$\log x=2\sum_{q=0}^\infty t^{2q+1}/(2q+1)$.  For $x=391/100$, one has
$t=291/491$.  Bounding the positive tail after the $t^5$ term by a geometric
series yields
\[
 \log\frac{391}{100}
 \leq2\left(t+\frac{t^3}{3}+\frac{t^5}{5}
 +\frac{t^7}{7(1-t^2)}\right)<\frac{273}{200}.
\]
For $x=391/1000$, one has $t=-609/1391$.  All terms in the series are
negative, and hence
\[
 \log\frac{391}{1000}
 <2\left(t+\frac{t^3}{3}+\frac{t^5}{5}\right)< -\frac{187}{200}.
\]
The final comparisons are exact rational inequalities.
\end{proof}

We finish by checking the constants used in \eqref{eq:hardy-lower}.  Since
$100/391<32/125$ and $1000/391<1279/500$, the preceding lemma gives
\begin{align*}
 C_L
 &<\frac14\left(\frac{273}{200}+\frac{32}{125}-1\right)
 =\frac{621}{4000}<\frac{39}{250},\\
 C_R
 &<\frac14\left(\frac{1279}{500}-\frac{187}{200}-1\right)
 =\frac{623}{4000}<\frac{39}{250}.
\end{align*}

\section*{Statement on AI usage}

We acknowledge the use of AI during the development of this work.
The authors first identified the annulus
\[
\{x\in\mathbb{R}^2:1/10<|x|<1\}
\]
as a candidate counterexample for \(m=3\). ChatGPT assisted in finding and
carrying out the explicit estimates needed to turn this candidate into a
rigorous counterexample, including suitable trial functions and exact
numerical certificates. The extension of the construction from \(m=3\) to
every \(m\geq4\), using auxiliary rectangles with suitably interlacing
Dirichlet eigenvalues and the resulting spectral-counting argument in
Section~\ref{sec:all-m}, was obtained with ChatGPT.

The accompanying Python/SymPy code for verifying the exact-arithmetic
computations~\cite{ExactArithmeticCode} was also generated by ChatGPT.
All AI-assisted mathematical arguments, computations, and code were
subsequently carefully checked by the authors. The authors take full
responsibility for all mathematical claims and computational results in
this paper.


\begin{thebibliography}{99}

\bibitem{AnoopBobkovDrabek2022}
T.~V. Anoop, V.~Bobkov, and P.~Dr\'abek,
\emph{Szeg\H{o}--Weinberger type inequalities for symmetric domains with holes},
SIAM J. Math. Anal. \textbf{54} (2022), no.~1, 389--422.
\href{https://doi.org/10.1137/21M1407227}
{doi:10.1137/21M1407227}.

\bibitem{Ashbaugh1999}
M.~S. Ashbaugh,
\emph{Open problems on eigenvalues of the Laplacian},
in T.~M. Rassias and H.~M. Srivastava (eds.),
\emph{Analytic and Geometric Inequalities and Applications},
Mathematics and Its Applications, vol.~478,
Kluwer Academic Publishers, Dordrecht, 1999, pp.~13--28.
\href{https://doi.org/10.1007/978-94-011-4577-0_2}{doi:10.1007/978-94-011-4577-0\_2}.

\bibitem{AshbaughBenguria1992}
M.~S. Ashbaugh and R.~D. Benguria,
\emph{A sharp bound for the ratio of the first two eigenvalues of Dirichlet
Laplacians and extensions},
Ann. of Math. (2) \textbf{135} (1992), no.~3, 601--628.
\href{https://doi.org/10.2307/2946578}{doi:10.2307/2946578}.

\bibitem{AshbaughBenguria1993}
M.~S. Ashbaugh and R.~D. Benguria,
\emph{Isoperimetric bounds for higher eigenvalue ratios for the
$n$-dimensional fixed membrane problem},
Proc. Roy. Soc. Edinburgh Sect. A \textbf{123} (1993), no.~6, 977--985.
\href{https://doi.org/10.1017/S0308210500029656}{doi:10.1017/S0308210500029656}.

\bibitem{AshbaughBenguria1994}
M.~S. Ashbaugh and R.~D. Benguria,
\emph{Bounds for ratios of eigenvalues of the Dirichlet Laplacian},
Proc. Amer. Math. Soc. \textbf{121} (1994), no.~1, 145--150.
\href{https://doi.org/10.1090/S0002-9939-1994-1186125-1}{doi:10.1090/S0002-9939-1994-1186125-1}.

\bibitem{AshbaughBenguriaLaugesenWeidl2009}
M.~S. Ashbaugh, R.~D. Benguria, R.~S. Laugesen, and T.~Weidl,
\emph{Low eigenvalues of Laplace and Schr\"odinger operators},
Oberwolfach Rep. \textbf{6} (2009), no.~1, 355--428.
\href{https://doi.org/10.4171/OWR/2009/06}{doi:10.4171/OWR/2009/06}.

\bibitem{BucurButtazzoHenrot2007}
D.~Bucur, G.~Buttazzo, and A.~Henrot,
\emph{Mini-workshop: Shape analysis for eigenvalues},
Oberwolfach Rep. \textbf{4} (2007), no.~2, 995--1026.
\href{https://doi.org/10.4171/OWR/2007/18}
{doi:10.4171/OWR/2007/18}.

\bibitem{Daners2003}
D.~Daners,
\emph{Dirichlet problems on varying domains},
J. Differential Equations \textbf{188} (2003), no.~2, 591--624.
\href{https://doi.org/10.1016/S0022-0396(02)00105-5}{doi:10.1016/S0022-0396(02)00105-5}.

\bibitem{Funano2026}
K.~Funano,
\emph{Some universal inequalities for Dirichlet eigenvalues of the Laplacian
on a Euclidean convex domain},
arXiv:2604.11114 [math.SP], 2026.
\href{https://doi.org/10.48550/arXiv.2604.11114}
{doi:10.48550/arXiv.2604.11114}.

\bibitem{GuptaMuldoon2000}
D.~P. Gupta and M.~E. Muldoon,
\emph{Riccati equations and convolution formulae for functions of Rayleigh type},
J. Phys. A: Math. Gen. \textbf{33} (2000), no.~7, 1363--1368.
\href{https://doi.org/10.1088/0305-4470/33/7/306}
{doi:10.1088/0305-4470/33/7/306}.

\bibitem{ExactArithmeticCode}
Y.~He, Q.~Tang, and H.~Zhang,
\emph{dirichlet-eigenvalue-ratio-verification},
Python/SymPy source code,
\url{https://github.com/QuanyuTang/dirichlet-eigenvalue-ratio-verification}.

\bibitem{IsmailMuldoon1995}
M.~E.~H. Ismail and M.~E. Muldoon,
\emph{Bounds for the small real and purely imaginary zeros of Bessel and
related functions},
Methods Appl. Anal. \textbf{2} (1995), no.~1, 1--21.
\href{https://doi.org/10.4310/MAA.1995.v2.n1.a1}
{doi:10.4310/MAA.1995.v2.n1.a1}.

\bibitem{Kishore1963}
N.~Kishore,
\emph{The Rayleigh function},
Proc. Amer. Math. Soc. \textbf{14} (1963), 527--533.
\href{https://doi.org/10.1090/S0002-9939-1963-0151649-2}
{doi:10.1090/S0002-9939-1963-0151649-2}.


\bibitem{DLMF}
F.~W.~J. Olver and L.~C. Maximon,
\emph{Bessel Functions}, Chapter~10 in
NIST Digital Library of Mathematical Functions.
\url{https://dlmf.nist.gov/10}.

\bibitem{PPW1956}
L.~E. Payne, G.~P\'olya, and H.~F. Weinberger,
\emph{On the ratio of consecutive eigenvalues},
J. Math. and Phys. \textbf{35} (1956), no.~1--4, 289--298.
\href{https://doi.org/10.1002/sapm1956351289}
{doi:10.1002/sapm1956351289}.



\end{thebibliography}
\end{document}